\documentclass{amsart}

\usepackage{amssymb,amsmath,graphicx}

\usepackage{color}

\newtoks\prt

\numberwithin{equation}{section}

\newtheorem{thm}{Theorem}[section]

\newtheorem{lemma}[thm]{Lemma}

\newtheorem{cor}[thm]{Corollary}

\newtheorem{example}[thm]{Example}

\theoremstyle{definition}

\def\eqn#1$$#2$${\begin{equation}\label#1#2\end{equation}}

\def\fra{\mathfrak{A}}

\def\A{\mathcal A}

\def\B{\mathcal B}

\def\C{\mathcal C}

\def\E{\mathcal E}

\def\R{\mathcal R}

\def\F{\mathcal F}

\def\G{\mathcal G}

\def\W{\mathcal W}

\def\e{e^\ast}

\def\E{E^\ast}

\def\H{\mathcal H}

\def\ep{\varepsilon}

\def\en{\mathbb N}

\def\er{\mathbb R}

\def\ov{\overline}

\def \Ch {\operatorname{Ch}}

\def \reg {\partial _{\kern1pt\text{reg}}}

\def\la{\langle}

\def\ra{\rangle}

\newcommand{\norm}[1]{\left\|#1\right\|}

\newcommand{\abs}[1]{\left| #1  \right|}

\begin{document}

\title[On the Banach-Mazur distance between continuous function spaces]
{On the Banach-Mazur distance between continuous function spaces with scattered boundaries}

\author{Jakub Rondo\v s}

\address{Charles University\\
Faculty of Mathematics and Physics\\
Department of Mathematical Analysis \\
Sokolovsk\'{a} 83, 186 \ 75\\Praha 8, Czech Republic}

\email{jakub.rondos@gmail.com}



\thanks{The research was supported by the Research grant GAUK 864120}

\begin{abstract}
We study the dependence of the Banach-Mazur distance between two subspaces of vector-valued continuous functions on the scattered structure of their boundaries. In the spirit of a result of Gordon \cite{Gordon3}, we show that the constant $2$ appearing in the Amir-Cambern theorem may be replaced by $3$ for some class of subspaces. This we achieve by showing that the Banach-Mazur distance of two function spaces is at least 3, if the height of the set of weak peak points of one of the spaces is larger than the height of a closed boundary of the second space. Next we show that this estimate can be improved, if the considered heights are finite and significantly different. As a corollary, we obtain new results even for the case of $\C(K, E)$ spaces.
\end{abstract}

\maketitle

\section{Introduction}

For a locally compact (Hausdorff) space $K$ and a real Banach space $E$, let $\C_0(K,E)$ denote the space of all continuous $E$-valued functions vanishing at infinity endowed with the sup-norm
\[\norm{f}_{\sup}=\sup_{x \in K} \norm{f(x)}, \quad f \in \C_0(K,E).\]
If $K$ is compact, then this space will be denoted by $\C(K,E)$. We write $\C_0(K)$ for $\C_0(K, \er)$ and $\C(K)$ for $\C(K, \er)$. All unexplained notions and definitions are contained in the next section.

We start with the following generalization of the well-known Banach-Stone theorem given independently by Amir \cite{amir} and Cambern \cite{cambern}. They showed that compact spaces $K_1$ and $K_2$ are homeomorphic if there exists an isomorphism $T\colon \C(K_1)\to\C(K_2)$ with $\norm{T} \norm{T^{-1}}<2$. Moreover, Amir conjectured that the number $2$ may be replaced by $3$. Cohen \cite{cohen-bound2} showed that this conjecture is not true in general. However, in \cite{Gordon3}, Gordon proved that it is true in the class of countable compact spaces: If $K_1$, $K_2$ are nonhomeomorphic countable compact spaces and $T: \C(K_1) \rightarrow \C(K_2)$ is an isomorphism, then $\norm{T}\norm{T^{-1}} \geq 3$.

The result of Gordon was extended in \cite[Theorem 1.5]{CandidoGalego3}, where the authors show that if $E$ is a Banach space having non-trivial cotype, and such that for every $n \in \en$, $E^n$ contains no subspace isomorphic to $E^{n+1}$, then countable compact spaces $K_1$ and $K_2$ are homeomorphic provided there exists an isomorphism $T:\C(K_1, E) \rightarrow \C(K_2, E)$ with $\norm{T}\norm{T^{-1}}<3$. It is clear that every finite-dimensional Banach space satisfies the above condition, and the authors also show in \cite[Remark 4.1]{CandidoGalego3} that there exist many infinite-dimensional Banach spaces that satisfy it. 

Starting in \cite{chuco}, and continuing in \cite{lusppams}, \cite{dosp}, and \cite{rondos-spurny-spaces}, the theorem of Amir and Cambern was extended to the context of subspaces. The final result for subspaces of scalar functions (see \cite[Theorem 1.1]{rondos-spurny-spaces}), reads as follows. For $i=1,2,$ let $\H_i \subseteq \C_0(K_i)$ be closed subspaces such that all points in their Choquet boundaries are weak peak points. If there exists an isomorphism $T\colon \H_1\to \H_2$ with $\norm{T} \norm{T^{-1}}<2$, then their Choquet boundaries $\Ch_{\H_i} K_i$ are homeomorphic (we recall that $x\in K_i$ is a \emph{weak peak point} (with respect to $\H_i$) if for a given $\ep\in (0,1)$ and a neighborhood $U$ of $x$ there exists a function $h\in B_{\H_i}$ such that $h(x)>1-\ep$ and $\abs{h}<\ep$ on $\Ch_{\H_i} K_i\setminus U$). Vector-valued extension of this result was given in \cite{rondos-spurny-vector-spaces}.

For a closed subspace $\H$ of $\C(K, E)$ we denote by $\W_{\H}$ the set of weak peak points of $K$ with respect to $\H$ (for the definition, which is similar to the one for scalar subspaces, see next section), and we consider the set
\[\Omega_{\H}=\{ y \in K, \exists g \in \H \colon g(y) \neq 0\}.\] 
The notion of the set $\Omega_{\H}$ will be useful, as among other things, it will allow us to work efficiently at the same with functions defined on compact spaces as well as on locally compact spaces. We also note that since $\Omega_{\H}$ is an open set in $K$, if $F^{(\alpha)} \cap \Omega_{\H}$ is finite for a subset $F$ of $K$ and an ordinal $\alpha$, then 
$F^{(\alpha+1)} \cap \Omega_{\H}$ is empty. A set $\B \subseteq K$ is a \emph{boundary} for $\H$ if for each $h \in \H$, $\norm{h}=\sup_{x \in \B} \norm{h(x)}$.

In view of the above results for subspaces, it seems natural to be looking for an extension of the result of Gordon to this context. We obtain such an extension as a corollary of the following theorem, that is inspired by \cite[Theorem 2.1]{CandidoGalego3} and \cite[Theorem 1.7]{GalegoVillamizar}. Also note that a similar result has been proved in \cite[Theorem 4]{VillamizarFelipe} for the case when $\H_1$ is an extremely regular subalgebra of $\C_0(K_1)$ and $\H_2=\C_0(K_2, E)$ (we recall that a subspace $\H \subseteq \C_0(K)$ is \emph{extremely regular} if for each $x \in K$, $U$ open neighbourhood of $x$ and $\ep>0$ there exists a function $h \in B_{\H}$ such that $h(x)=1$ and $\abs{h}<\ep$ on $K \setminus U$).

\begin{thm}
	\label{jednostranna}
	For $i=1, 2$, let $\H_i$ be a closed subspace of $\C(K_i, E_i)$, where $K_i$ is a compact Hausdorff space and $E_i$ is a Banach space. Let $E_2$ do not contain an isomorphic copy of $c_0$, and let $\B$ be a closed boundary for $\H_2$. Suppose that there exists an into isomorphism $T:\H_1 \rightarrow \H_2$ with $\norm{T}\norm{T^{-1}}<3$. Then:
	\begin{itemize}
		\item[a)] If $\alpha$ is an ordinal such that $\W_{\H_1}^{(\alpha)}$ is nonempty, then so is $\B^{(\alpha)} \cap \Omega_{\H_2}$.
		\item[b)] If $\alpha$ is an ordinal such that $\B^{(\alpha)} \cap \Omega_{\H_2}$ is finite, then so is $\W_{\H_1}^{(\alpha)}$.
		\item[c)] Let $\alpha$ be an ordinal such that $\B^{(\alpha)} \cap \Omega_{\H_2}$ is finite. Then $E_2^m$ contains an isomorphic copy of $E_1^n$, where $\abs{\W_{\H_1}^{(\alpha)}}=n$, $\abs{\B^{(\alpha)} \cap \Omega_{\H_2}}=m$. In particular, if $E_1=E_2$ is a finite-dimensional space, then $n \leq m$.
		\item[d)] Let $\H_1$ contains the constant functions. Then for an ordinal $\alpha$ such that $\B^{(\alpha)} \cap \Omega_{\H_2}$ is infinite it holds that 
		$\abs{\W_{\H_1}^{(\alpha)}} \leq \abs{\B^{(\alpha)} \cap \Omega_{\H_2}}$.
	\end{itemize}
\end{thm}

We point out that in Theorem \ref{jednostranna}, the assumption of weak peak points is imposed only on one of the two considered function spaces. This may be surprising, as it is known that the Amir-Cambern theorem for subspaces fails completely, if the Choquet boundary of at least one of the two function spaces does not consist of weak peak points. Indeed, in \cite{hess}, Hess shows that for each $\ep>0$, there exists a space $\H_2 \subseteq \C([0, \omega])$ with $\Ch_{\H_2} K_2=[0, \omega)$, such that there exists an isomorphism $T: \H_1=\C([0, \omega]) \rightarrow \H_2$ with $\norm{T}\norm{T^{-1}}<1+\ep$. Here, of course, the space $\H_2$ does not satisfy the assumption of weak peak points. The reason why this example does not contradict Theorem \ref{jednostranna} is that in this case $\Omega_{\H_2}=[0, \omega]$, and the smallest closed boundary for $\H_2$ is $\B=\ov{\Ch_{\H_2} K_2}=[0, \omega]$, thus
\[\W_{\H_1}=[0, \omega]=\B \cap \Omega_{\H_2}.\]
Theorem \ref{jednostranna} shows that if the topological differences of the set of weak peak points of one space and a boundary of the other space are more significant than in this case, then such a thing is no longer possible. This example also shows that in Theorem \ref{jednostranna}, we cannot omit the assumption that the boundary is closed.

The following theorem extends \cite[Theorem 1.5]{CandidoGalego3}.

\begin{thm}
	\label{homeomorfni}
	For $i=1, 2$, let $K_i$ is a compact Hausdorff space, $E_i$ is a Banach space not containing an isomorphic copy of $c_0$, and $\H_i$ is a closed subspace of $\C(K_i, E_i)$ such that $\Ch_{\H_i} K_i$ is closed and countable, and consists of weak peak points. Let for each natural numbers $m<n$, $E_1^m$ does not contain an isomorphic copy of $E_2^n$, and $E_2^m$ does not contain an isomorphic copy of $E_1^n$. If there exists an isomorphism $T: \H_1 \rightarrow \H_2$ with $\norm{T}\norm{T^{-1}}<3$, then $\Ch_{\H_1} K_1$ is homeomorphic to $\Ch_{\H_2} K_2$.
\end{thm}

\begin{proof}
	Since $\Ch_{\H_1} K_1$ and $\Ch_{\H_2} K_2$ are countable compact Hausdorff spaces, there exist nonzero ordinals $\alpha, \beta$ and $m, n \in \en$ such that 
	$\Ch_{\H_1} K_1$ si homeomorphic to $[1, \omega^{\alpha}m]$ and $\Ch_{\H_2} K_2$ is homeomorphic to $[1, \omega^{\beta}n]$, see \cite[Proposition 8.6.5]{semadeni}. By using Theorem \ref{jednostranna} a) with $\W_{\H_1}=\Ch_{\H_1} K_1$, $\B=\Ch_{\H_2} K_2\subseteq \Omega_{\H_2}$, we obtain that $\alpha \leq \beta$. Thus $\alpha=\beta$ by symmetry. By using c) of Theorem \ref{jednostranna} twice it follows that also $m$ is equal to $n$, which concludes the proof. 
\end{proof}

If $\H$ is an extremely regular subspace of $\C(K)$, then it is clear that each point of $K$ is a weak peak point with respect to $\H$, and thus $\Ch_{\H} K=K$, see \cite[Lemma 2.5]{rondos-spurny-vector-spaces}. Thus extremely regular spaces serve as a simple example of spaces that have closed boundaries consisting of weak peak points.

It is natural to ask whether the assumption that the boundaries are closed is necessary in Theorem \ref{homeomorfni} (we remind that in the case of the Amir-Cambern theorem for subspaces, no topological assumptions need to be imposed on the boundaries, see \cite[Theorem 1.1]{rondos-spurny-spaces} or \cite[Theorem 1.1]{rondos-spurny-vector-spaces}). In section 3, we present a simple example showing that Theorem \ref{homeomorfni} does not hold without the assumption of closed Choquet boundaries (see Example \ref{Priklad}). 

Next we turn our attention to isomorphisms that are not necessarily bounded by the number $3$. It is well known that if two $\C(K)$ spaces are isomorphic, then the underlying compact spaces have the same cardinality, see \cite{cengiz} for the case of scalar functions and \cite{GalegoVillamizar} for an analogous result for vector-valued functions. Generalizations of those results to the context of subspaces were given in \cite{rondos-spurny-spaces} and \cite{rondos-spurny-vector-spaces}. 

Also, it is known that isomorphisms between $\C(K)$ spaces are somehow connected with the scattered structures of the underlying compact spaces. Indeed, it was proved in \cite[Theorem 1.4]{CandidoScattered} that if $E$ is a Banach space not containing an isomorphic copy of $c_0$, $K_2$ is a scattered locally compact space and $\C_0(K_1)$ embeds isomorphically into $\C_0(K_2, E)$, then $K_1$ is also scattered. 

Moreover, there have been proven estimates of the Banach-Mazur distance of $\C(K)$ spaces from $\C_0(\Gamma, E)$ spaces, where $\Gamma$ is a discrete set, and from $\C(F)$, where $F$ is a compact space of height $2$ (in particular for $F=[0, \omega]$), based on the height of the compact space $K$. 
It was proved in \cite[Theorem 1.2]{CandidoGalegoComega} that if $K$ is a compact space with $K^{(n)} \neq \emptyset$ for some $n \in \en$, $F$ is a compact space with $F^{(2)}=\emptyset$ and there exists an isomorphism $T: \C(K) \rightarrow \C(F)$, then $\norm{T}\norm{T^{-1}} \geq 2n-1$.  Moreover, if $\abs{K^{(n)}}>\abs{F^{(1)}}$, then $\norm{T}\norm{T^{-1}} \geq 2n+1$. In \cite[Theorem 1.1]{CANDIDOc0} it has been showed that if $\Gamma$ is an infinite discrete space, $E$ is a Banach space not containing an isomorphic copy of $c_0$ and $T: \C(K) \rightarrow \C_0(\Gamma, E)$ is an into isomorphism, then for each $n \in \en$, if $K^{(n)} \neq \emptyset$, then $\norm{T}\norm{T^{-1}} \geq 2n+1$. 
Similar results for isomorphisms with range in $\C_0(\Gamma, E)$ spaces were proven before in \cite{CandidoGalegoFund} and \cite{CandidoGalegoCotype}. In the following theorem, which generalizes the above results, we show that the estimate $3$ for the Banach-Mazur distance of two function spaces obtained in Theorem \ref{jednostranna} may be improved, if the heights of the set of weak peak points and of the closed boundary are finite and the height of the set of weak peak points is significantly larger.

\begin{thm}
	\label{distance}
	For $i=1, 2$, let $K_i$ be a compact space, $E_i$ be a Banach space and $\H_i$ be a closed subspace of $\C(K_i, E_i)$. Let $E_2$ do not contain an isomorphic copy of $c_0$, $\B$ be a closed boundary for $\H_2$ and $T: \H_1 \rightarrow \H_2$ be an into isomorphism. Let $l, k \in \en$, $l \geq k$. Let one of the following conditions hold.
	\begin{itemize}
		\item[(i)]  $\W_{\H_1}^{(l)}$ is nonempty and $\B^{(k)} \cap \Omega_{\H_2}$ is empty. 
		\item[(ii)] $\W_{\H_1}^{(l)}$ is infinite and $\B^{(k)} \cap \Omega_{\H_2}$ is finite.
		\item[(iii)]  $\W_{\H_1}^{(l)}$ and $\B^{(k)} \cap \Omega_{\H_2}$ are both finite, and $E_2^m$ does not contain an isomorphic copy of $E_1^n$, where $\abs{\W_{\H_1}^{(l)}}=n$, $\abs{\B^{(k)} \cap \Omega_{\H_2}}=m$. 	
	\end{itemize}
	Then $\norm{T}\norm{T^{-1}} \geq \max\{3, \frac{2l+2-k}{k}\}$.
\end{thm}

We show that the lower bounds that we obtain here are the same as in \cite{CandidoGalegoComega} and \cite{CANDIDOc0}. Thus let $K_1$ be a compact space with $K^{(n)}$ nonempty and $\H_1=\C(K_1)$. Then $\W_{\H_1}=K_1$ by the Urysohn Lemma. If $\H_2=\C(F)$, where $F^{(2)}=\emptyset$, then let $\B=F$. Since $K^{(n-1)}$ is infinite and $\B^{(1)}$ is finite, it follows that from (ii) with $l=n-1$ and $k=1$ that we obtain the same lower bound as in \cite[Theorem 1.2]{CandidoGalegoComega}. From (iii) with $E_1=E_2=\er$, $l=n$ and $k=1$, we also obtain the same bound in the case when $\abs{K^{(n)}}>\abs{F^{(1)}}$. Next, let $\H_2=\C_0(\Gamma, E)$, where $\Gamma$ is a discrete space. Let $K_2=\Gamma \cup \{\alpha\}$ be the one-point compactification of $\Gamma$. Then we may put $\B=K_2$, $k=1$, $l=n$, and then $\B^{(1)} \cap \Omega_{\H_2}=\{\alpha\} \cap \Omega_{\H_2}=\emptyset$, and from (i) we obtain the same estimate as in \cite[Theorem 1.1]{CANDIDOc0}.

From the fact that the term $\frac{2l+2-k}{k}$ appearing in Theorem \ref{distance} can happen to be arbitrary close to $1$, in contrast with the number $3$, it is clear that this estimate is in general far from being optimal. Nevertheless, the estimate might be reasonable in cases when $k$ is substantially smaller than $n$. For $k=1$ it is optimal, since the number is attained e.g. for $\H_1=\C([0, \omega^n], E)$, $\H_2=\C_0(\omega, E)$, where $E$ is a Banach space not containing an isomorphic copy of $c_0$, see \cite[Theorem 1.2]{CANDIDOc0}. Moreover, Theorem \ref{distance} has the following simple, but interesting consequence.

\begin{cor}
\label{finitefinite}
	For $i=1, 2$, let $K_i$ be a compact space, $E_i$ be a Banach space and $\H_i$ be a closed subspace of $\C(K_i, E_i)$. Let $E_2$ do not contain an isomorphic copy of $c_0$, $\B$ be a closed boundary for $\H_2$, and suppose that $\H_1$ is isomorphic to a subspace of $\H_2$. Then if $\B^{(k)} \cap \Omega_{\H_2}$ is empty for some $k \in \en$, then $ht(\W_{\H_1})$ is finite.
\end{cor}

\begin{proof}
	Let $T: \H_1 \rightarrow \H_2$ be an into isomorphism. Then by Theorem \ref{distance}, $\W_{\H_1}^{(l)}$ is empty for \[l>\frac{\norm{T}\norm{T^{-1}}k+k-2}{2}.\] Thus $ht(\W_{\H_1})$ is finite.	
\end{proof}

Notice that all the above results hold automatically also for subspaces $\H \subseteq \C_0(K, E)$, where $K$ is locally compact. Indeed, by \cite[Lemma 2.10]{rondos-spurny-vector-spaces}, if $J=K \cup \{\alpha\}$ is the one-point compactification of $K$, then $\H$ is isometric to a subspace $\widetilde{\H}$ of $J$, defined by
	\[
\widetilde{\H}=\{h\in \C(J,E)\colon h|_{K}\in \H\ \&\ h(\alpha)=0\}.
\]
Moreover, $\Ch_{\H} K$ is homeomorphic to $\Ch_{\widetilde{\H}} J$ and a point $x \in \Ch_{\H} K$ is a weak peak point with respect to $\H$ if and only if it is a weak peak point with respect to $\widetilde{\H}$. Also, it is clear that if $\B$ is a closed boundary for $\H$, then $\widetilde{\B}=\B \cup \{\alpha\}$ is a closed boundary for $\widetilde{\H}$, and $\widetilde{\B} \cap \Omega_{\widetilde{\H}}=(\B \cup \{\alpha\}) \cap \Omega_{\widetilde{\H}}=\B \cap \Omega_{\H}$. Thus the results valid for $\widetilde{\H}$ apply also on $\H$.

Next, we point out what the above results give for the case of $\C(K, E)$ spaces.

\begin{cor}
\label{Ckacka}
Let $K_1, K_2$ be infinite compact spaces, $E_1, E_2$ be Banach spaces, $E_2$ not containing an isomorphic copy of $c_0$. 
\begin{itemize}
	\item[(i)] If $l, k \in \en$, $l>k$, $ht(K_1)=l$, $ht(K_2)=k$, and there exists an into isomorphism $T:\C(K_1, E_1) \rightarrow \C(K_2, E_2)$, then \[\norm{T}\norm{T^{-1}} \geq \max\{3, \frac{2l-k-1}{k-1}\}.\]
	\item[(ii)]
	If $\C(K_1, E_1)$ embeds isomorphically into $\C(K_2, E_2)$ and $ht(K_2)$ is finite, then $ht(K_1)$ is finite.
\end{itemize} 
\end{cor}

\begin{proof}
By the Urysohn Lemma, it is clear that each point of $K_1$ is a weak peak point with respect to $\C(K_1, E_1)$, and consequently, $\Ch_{\H_1} K_1=K_1$, see \cite[Lemma 2.5]{rondos-spurny-vector-spaces}. We denote $\B=K_2$. Since $K_2^{(k-1)}$ is finite, and $K_1^{(l-2)}$ is infinite, with the use of Theorem \ref{distance}(ii) we obtain the estimates in the statement (i) (the bound $3$ was known before, see \cite[Theorem 2.1]{CandidoGalego3}). The statement (ii) follows in the same way as Corollary \ref{finitefinite}.
\end{proof}

\section{Auxiliary results}

This section contains the unexplained definitions and notations, as well as several simple results that will be frequently used later.

To start with, the cardinality of a set $M$ is denoted by $\abs{M}$. Next, the \emph{derivative} of a topological space $S$ is defined recursively as follows. The set $S^{(1)}$ is the set of accumulation points of $S$, and for an ordinal $\alpha>1$, let $S^{(\alpha)}=(S^{\beta})^{(1)}$, if $\alpha=\beta+1$, and $S^{(\alpha)}=\bigcap_{\beta<\alpha} S^{(\beta)}$, if $\alpha$ is a limit ordinal. Moreover, let $S^{(0)}=S$. The topological space $S$ is called \emph{scattered} if there exists an ordinal $\alpha$ such that $S^{(\alpha)}$ is empty, and minimal such $\alpha$ is called the \emph{height} of $S$ and is denoted by $ht(S)$. Thus if $K$ is a scattered compact space, then $ht(K)$ is a successor ordinal, $K^{(ht(K)-1)}$ is finite, and $K^{(\alpha)}$ is infinite for each ordinal $\alpha<ht(K)-1$. All topological spaces are assumed to be Hausdorff.

All Banach spaces are tacitly assumed to be real and of dimension at least $1$. If $E$ is a Banach space then $E^*$ stands for its dual space. We denote by $B_E$ and $S_E$ the unit ball and sphere in $E$, respectively, and we write $\la\cdot,\cdot\ra\colon E^*\times E\to\er$ for the duality mapping. All isomorphisms between Banach spaces are assumed to be surjective, otherwise they are referred to as into isomorphisms. The \emph{Banach-Mazur distance} of Banach spaces $E_1$, $E_2$ is defined to be the infimum of $\norm{T}\norm{T^{-1}}$ over the set of all isomorphisms $T:E_1 \rightarrow E_2$ and is denoted by $d_{BM}(E_1, E_2)$. 

Let $K$ be a locally compact space and $E$ be a Banach space. For $h \in \C_0(K, E)$ and $\e \in \E$, $\e(h)$ is the element of $\C_0(K)$ defined by $\e(h)(x)=\la \e, h(x) \ra$ for $x \in K$. Next, for a function $f \in \C_0(K)$ and $e \in E$, the function $f \otimes e \in \C_0(K, E)$ is defined by
\[ (f \otimes e)(x)=f(x)e, \quad x \in K .\]

If $\H$ is a closed subspace of $\C_0(K, E)$, then its canonical scalar function space $\A \subseteq \C_0(K)$ is defined as the closed linear span of the set 
\[ \lbrace \e(h): \e \in \E, h \in \H \rbrace  \subseteq \C_0(K).\]
Moreover, let $\A^E$ stands for the subspace of $\A$ consisting of all functions $h$ from $\A$ satisfying that $h \otimes e \in \H$ for each $e \in E$. The \emph{Choquet boundary} $\Ch_{\H} K$ of $\H$ is defined as the Choquet boundary of $\A$, that is, $\Ch_{\H} K$ is the set of those points $x \in K$ such that the functional 
\[i(x): h \mapsto h(x), \quad h \in \A\] 
is an extreme point of $B_{\A^{*}}$. It follows by \cite[Lemma 2.1]{rondos-spurny-vector-spaces} and \cite[Theorem 2.3.8]{fleming-jamison-1} that the Choquet boundary is a boundary for $\H$, that is,
\[\norm{h}=\sup_{x \in \Ch_{\H} K} \norm{h(x)}, \quad h \in \H.\]
Thus a canonical example of a closed boundary for $\H$ is the closure $\ov{\Ch_{\H} K}$ of the Choquet boundary of $\H$. Moreover, the following strict maximum principle holds. If $h$ is a function from $\A$ and $C \in \er$, then $h<C$ on $K$ provided $h<C$ on $\Ch_{\H} K$, see \cite[Proposition 3.87]{lmns}. Also, it is easy to see that the Choquet boundary $\Ch_{\H} K$ is a subset of $\Omega_{\H}$. Indeed, the space $\A$ is nontrivial by our assumption, and hence $0$ is not an extreme point of $\B_{\A^*}$.

In this paper, we use a slightly stronger definition of weak peak points for vector-valued subspaces than in \cite{rondos-spurny-vector-spaces}. Here we say that a point $x \in K$ is a \emph{weak peak point} (with respect to $\H$), if for each neighbourhood $U$ of $x$ and $\ep \in (0, 1)$ there exists a function $h \in \A^E$, such that 
\[0 \leq h \leq 1, h(x)>1-\ep \text{ and } h<\ep \text{ on } \Ch_{\H} K \setminus U.\] 
Thus if $\H=\C(K, E)$ then $\A^E=\A=\C(K)$, and each point of $K$ is a weak peak point. If $X$ is a compact convex set in a locally convex space and $\H=\fra(X, E)$, the space of affine $E$-valued continuous functions on $X$, then $\A^E=\A=\fra(X, \er)$, and a point $x \in X$ is a weak peak point with respect to $\H$ if and only if it is a weak peak point in the classical sense. We also note that the set $\W_{\H}$ of weak peak points of $\H$ is a subset of the Choquet boundary of $\H$, see \cite[Lemma 2.5]{rondos-spurny-vector-spaces}.

For the proof of the following Lemma, we recall that a series $\sum_{i=1}^{\infty} e_i$ in a Banach space $E$ is \emph{weakly unconditionally Cauchy} if $\sum_{i=1}^{\infty} \abs{\la \e, e_i \ra}< \infty$ for each $\e \in \E$.

\begin{lemma}
	\label{c_0}
	Let for $i=1, 2$, $K_i$ is a compact space, $E_i$ is a Banach space and $\H_i$ is a closed subspace of $\C(K_i, E_i)$. Let $\A_1$ be the scalar function space associated to $\H_1$, let $T:\H_1 \rightarrow \H_2$ be an into isomorphism, and fix a nonzero vector $e \in E_1$. Suppose that $\{ f_i \}$ is a sequence in $\A_1$ satisfying that $h_i=f_i \otimes e \in \H_1$ for each $i \in \en$. Suppose that there exists a constant $C \in \er$ satisfying that for each $n \in \en$ and $\alpha_1, \ldots , \alpha_n \in S_{\er}$ it holds that
	\[\norm{\sum_{i=1}^n \alpha_i h_i}<C.\]
	Moreover, let $y \in K_2$ and $\ep>0$ satisfy that 
	$\norm{Th_i(y)}>\ep$ for each $i \in \en$. Then $E_2$ contains an isomorphic copy of $c_0$.
\end{lemma}

\begin{proof}
	We consider the evaluation mapping $\phi\colon K_2 \times \E_2 \to \H_2^*$ defined as
	\[ \la \phi(y, \e), g \ra=\la \e, g(y) \ra, \quad g \in \H_2, y \in K_2, \e \in \E_2.\] 
	By the classical characterization of the Banach spaces containing $c_0$, see \cite[Theorem 6.7]{morrison2001functional}, it is enough to show that the series $\sum_{i=1}^{\infty} Th_i(y)$ is weakly unconditionally Cauchy in $E_2$. Thus fix $e^* \in S_{E_2^*}$, and let $T^*$ be the adjoint of $T$. 
	Fix $n \in \en$, and let $\alpha_1, \ldots, \alpha_n \in S_{\er}$ satisfy 
	\[\abs{\la T^*\phi(y, e^*), h_i \ra}=\alpha_i \la T^*\phi(y, e^*), h_i \ra, \quad i=1, \ldots, n.\]
	Then we have 
	\begin{equation}
	\nonumber
	\begin{aligned}
	&\sum_{i=1}^{n} \abs{\la e^*, Th_i(y) \ra}=\sum_{i=1}^{n} \abs{\la \phi(y, e^*), Th_i \ra}=
	\sum_{i=1}^{n} \abs{\la T^*\phi(y, e^*), h_i \ra}
	=\\&=
	\sum_{i=1}^{n} \alpha_i \la T^*\phi(y, e^*), h_i \ra=\la T^*\phi(y, e^*), \sum_{i=1}^{n} \alpha_i h_i \ra 
	\leq \\& \leq
	\norm{T^*\phi(y, e^*)} \norm{\sum_{i=1}^n \alpha_i h_i}\leq \norm{T^*} \norm{e^*}{\norm{\sum_{i=1}^n \alpha_i h_i}} < C \norm{T^*}. 
	\end{aligned}
	\end{equation}
	Thus also $\sum_{i=1}^{\infty} \abs{\la e^*, Th_i(y) \ra} \leq C\norm{T^*}<\infty$, which finishes the proof.
\end{proof}

We also prove the following two simple observations, as we will be using them repeatedly later on.

\begin{lemma}
\label{omezenost}
Let $K$ be a compact space, $E$ be a Banach space and $\H$ be a closed subspace of $\C(K, E)$. Let $n \in \en$, $C \in \er$, $\{\alpha_i\}_{i=1}^n$ be positive numbers, $\{U_i\}_{i=1}^n$ be pairwise disjoint nonempty open sets in $K$, and let functions $\{h_i\}_{i=1}^n \subseteq \H$ satisfy that for each $i=1 \ldots, n$, $\norm{h_i} \leq C$ on $K$ and $\norm{h_i}<\alpha_i$ on $\Ch_{\H} K \setminus U_i$. Then
\[\norm{\sum_{i=1}^n h_i} \leq C+\sum_{i=1}^n \alpha_i.\]
\end{lemma}

\begin{proof}
If $x \in U_{i_0} \cap \Ch_{\H} K$ for some $i_0 \in \{1, \ldots, n\}$, then since the sets $U_i$ are pairwise disjoint, 
\[\norm{\sum_{i=1}^n h_i(x)} \leq \sum_{i=1}^n \norm{h_i(x)} \leq C+\sum_{i \neq \i_0} \alpha_i.\]
If $x \in \Ch_{\H} K \setminus \bigcup_{i=1}^n U_i$, then 
\[\norm{\sum_{i=1}^n h_i(x)} \leq \sum_{i=1}^n \norm{h_i(x)} \leq \sum_{i=1}^n \alpha_i.\]
Thus $\norm{\sum_{i=1}^n h_i(x)} \leq C+\sum_{i=1}^n \alpha_i$ for $x \in \Ch_{\H} K$, and hence $\norm{\sum_{i=1}^n h_i(x)} \leq C+\sum_{i=1}^n \alpha_i$ for $x \in K$.
\end{proof}

\begin{lemma}
	\label{nerovnost}
	Let $K$ be a compact space, $\A \subseteq \C(K)$ be a closed subspace, $0<\ep<1$, and functions $h, g \in \A$ satisfy that $0 \leq g \leq 1, 0 \leq h \leq 1$ on $K$, $h>1-\ep$ on an open set $U$ and $g<\ep$ on $\Ch_{\A} K \setminus U$. Then $g<h+\ep$ on $K$.
\end{lemma}

\begin{proof}
	If $x \in \Ch_{\A} K \cap U$, then \[g(x) \leq 1=1-\ep+\ep<h(x)+\ep.\] On the other hand, if $x \in \Ch_{\A} K \setminus U$, then	
	\[g(x)<\ep \leq h(x)+\ep.\]
	Thus $g-h<\ep$ on $\Ch_{\A} K$, and hence $g-h<\ep$ on $K$.
\end{proof}

\section{Isomorphisms with bound 3}

Throughout this section, let for $i=1, 2$, $K_i$ be a compact space, $E_i$ be a Banach space and $\H_i$ be a closed subspace of $\C(K_i, E_i)$. Let $\A_1$ be the scalar function space associated to $\H_1$ and $\B$ be a closed boundary for $\H_2$.
For $g \in \H_2$ and $\ep>0$ we denote 
\[\R(g, \ep)=\{ y \in K_2 \colon \norm{g(y)} \geq \ep\}.\]
Notice that from the definition it is clear that for each function $g \in \H_2$, $\R(g, \ep) \subseteq \Omega_{\H_2}$.
Next, inspired by \cite{CandidoGalego3}, we define certain sets $\F$, $\Lambda$, which we use for the proof of Theorem \ref{jednostranna}. For $0 < \ep <1$, functions $f_1, \ldots, f_n \in \A_1^{E_1}$ and points $z_1, \ldots, z_n \in \Ch_{\H_1} K_1$ we denote 
\begin{equation}
\nonumber
\begin{aligned}
\mathcal{F}((f_i, z_i)_{i=1}^n, \ep)
=&\{(g_1, \ldots, g_n) \in (\A_1^{E_1})^n \forall i=1, \ldots, n : 0 \leq g_i \leq 1, \\& g_i(z_i)>1-\ep \text{ and } g_i <f_i+\ep\}.
\end{aligned}
\end{equation}

Suppose moreover that $T:\H_1 \rightarrow \H_2$ is an into isomorphim, $e_1, \ldots, e_n \in E_1$ are nonzero vectors with $\max_{i=1, \ldots , n} \norm{e_i}=1$, $U_1, \ldots, U_n$ are pairwise disjoint nonempty open sets in $K_1$, and let $h_1, \ldots, h_n$ be functions in $\A_1^{E_1}$ satisfying that 
\[0 \leq h_i \leq 1, \norm{h_i}>1-\ep \text{ and } h_i<\ep \text{ on } \Ch_{\H_1} K_1 \setminus U_i, \quad i=1, \ldots, n.\]
Then for points $x_1, \ldots, x_n \in \Ch_{\H_1} K_1$ satisfying that $h_i(x_i)>1-\ep$ for each $i=1, \ldots, n$, we denote
\begin{equation}
\nonumber
\begin{aligned}
\Lambda_{(T, (h_i, e_i)_{i=1}^n, \ep)}((x_i)_{i=1}^n)
=\B \cap \bigcap_{(g_1, \ldots, g_n) \in \mathcal{F}((h_i, x_i)_{i=1}^n, \ep)} \R(T(\sum_{i=1}^n g_i \otimes e_i), \ep).
\end{aligned}
\end{equation}

Notice that the set $\mathcal{F}((h_i, x_i)_{i=1}^n, \ep)$ is nonempty, as it contains $(h_1, \ldots, h_n)$. The following technical Lemma, describing the properties of the above defined set $\Lambda$, contains the main part of the proof of Theorem \ref{jednostranna}.

\begin{lemma}
	\label{main}
	Let $E_2$ do not contain an isomorphic copy of $c_0$. Let $n \in \en$, and suppose that there exists an into  isomorphism $T: \H_1 \rightarrow \H_2$ with $\norm{T}<3$ and $\norm{T^{-1}}=1$. Fix	
	 \[0<\ep<\frac{3-\norm{T}}{\norm{T}(7n+2)+3n+2}.\]
	 Moreover, let $e_1, \ldots, e_n \in E_1$ be nonzero vectors with $\max_{i=1, \ldots , n} \norm{e_i}=1$, $U_1, \ldots, U_n$ be pairwise disjoint nonempty open subsets of $K_1$, and let $h_1, \ldots, h_n$ be functions in $\A_1^{E_1}$ satisfying that 
	 \[0 \leq h_i \leq 1, \norm{h_i}>1-\ep \text{ and } h_i<\ep \text{ on } \Ch_{\H_1} K_1 \setminus U_i, \quad i=1, \ldots, n.\]
	
	Then:
	\begin{itemize}
		\item [a)] If for $i=1, \ldots, n$, $\{x_i^k\}_{k \in \en} \subseteq \W_{\H_1}$ are infinite sequences of pairwise distinct points satisfying that 
		\[h_i(x_i^k)>1-\ep, \quad i=1, \ldots, n, k \in \en,\]
		 then the intersection $\bigcap_{k \in \en} \Lambda_{(T, (h_i, e_i)_{i=1}^n, \ep)} ((x_i^k)_{i=1}^n)$ is empty.
		
		\item[b)] If $x_1, \ldots, x_n \in \W_{\H_1}^{(\alpha)}$ for an ordinal $\alpha$ are points satisfying that
		\[h_i(x_i)>1-\ep, \quad i=1, \ldots, n,\]
		then $\Lambda_{(T, (h_i, e_i)_{i=1}^n, \ep)}((x_i)_{i=1}^n) \cap \B^{(\alpha)} \neq \emptyset$.
	\end{itemize}
\end{lemma}

\begin{proof}
	For the proof of a), suppose that for $i=1, \ldots, n$ there exit such sequences $\{x_i^k\}_{k \in \en}$ and $y \in \bigcap_{k \in \en} \Lambda_{(T, (h_i, e_i)_{i=1}^n, \ep)} ((x_i^k)_{i=1}^n)$. Passing to subsequences and using the continuity of the functions $h_i$, we may assume that there exist pairwise disjoint open sets $\{V_i^k\}_{k \in \en}, i=1, \ldots, n$, each containing $x_i^k$, and such that $h_i>1-\ep$ on $V_i^k$. For each $i$ and $k$, since $x_i^k$ is a weak peak point, we may find a function $g_i^k \in \A_1^{E_1}$ satisfying that
	\[0 \leq g_i^k \leq 1, g_i^k(x_i^k)>1-\frac{\ep}{2^k} \text{ and } g_i^k<\frac{\ep}{2^k} \text{ on } \Ch_{\H_1} K_1 \setminus V_i^k.\]
	Notice that for each $i=1, \ldots, n$, since the sets $\{V_i^k\}_{k \in \en}$ are pairwise disjoint, for each $m \in \en$ and $\alpha_1, \ldots , \alpha_m \in S_{\er}$ it holds that $\norm{\sum_{k=1}^m \alpha_k g_i^k}\leq 1+\ep$
	by Lemma \ref{omezenost}. Also, for each $i, k$, we have $g_i^k<h_i+\ep$ (see Lemma \ref{nerovnost}), thus for each $k \in \en$ it holds that $(g_1^k, \ldots, g_n^k) \in \F((h_i, x_i^k)_{i=1}^n, \ep)$. Hence $\norm{T(\sum_{i=1}^n (g_i^k \otimes e_i))(y)} \geq \ep$ by the definition of $\Lambda_{(T, (h_i, e_i)_{i=1}^n, \ep)}((x_i^k)_{i=1}^n)$.
	
	Consequently, for each $k \in \en$ there exists $i \in \{1, \ldots, n\}$ such that 
	\[\norm{T(g_i^k \otimes e_i)(y)}\geq \frac{\ep}{n}.\] Thus there exists $i_0 \in \{1, \ldots, n\}$ such that the set
	\[N=\{ k \in \en: \norm{T(g_{i_0}^k \otimes e_{i_0})(y)}\geq \frac{\ep}{n}\}\] is infinite. But since $E_2$ does not contain a copy of $c_0$, the sequence $\{g_{i_0}^k\}_{k \in N}$ contradicts Lemma \ref{c_0}, which finishes the proof of a).
	
	For the proof of b) we proceed by transfinite induction. For $\alpha=0$ we want to prove that $\Lambda_{(T, (h_i, e_i)_{i=1}^n, \ep)}((x_i)_{i=1}^n)$ is a nonempty set for each points $x_1, \ldots, x_n \in \W_{\H_1}$ satisfying $h_i(x_i)>1-\ep$ for each $i=1, \ldots, n$. Notice that the set  \\$\Lambda_{(T, (h_i, e_i)_{i=1}^n, \ep)}((x_i)_{i=1}^n)$ is compact, as it is an intersection of closed subsets of $\B$. Thus it is enough to show that the collection of sets
	\[\{y \in \B: \norm{T(\sum_{i=1}^n g_i \otimes e_i)(y)} \geq \ep \}, \quad (g_i)_{i=1}^n \in \mathcal{F}((h_i, x_i)_{i=1}^n, \ep) \] 
	has the finite intersection property. So, let functions $g_i^k \in \A_1^{E_1}$, $i=1, \ldots, n$, $k=1, \ldots, p$ satisfy
	\[0 \leq g_i^k \leq 1, g_i^k(x_i)>1-\ep \text{ and } g_i^k <h_i+\ep.\]
	By the continuity of the functions $g_i^k$, for each $i=1, \ldots, n$ there exists an open set $V_i \subseteq U_i$ containing $x_i$, and such that $g_i^k>1-\ep$ on $V_i$ for each $k=1, \ldots, p$. Since $x_i$ is a weak peak point, there exists a function $f_i \in \A_1^{E_1}$ satisfying that 
	\[0 \leq f_i \leq 1, f_i(x_i)>1-\ep \text{ and } f_i <\ep \text{ on } \Ch_{\H_1} K_1 \setminus V_i.\]
	Then $f_i<g_i^k+\ep$ for each $i$ and $k$ by Lemma \ref{nerovnost}. 
	
	Next, fix $i_0 \in \{1, \ldots, n\}$ satisfying that $\norm{e_{i_0}}=1$. We have 
	\[\norm{((2f_{i_0}+h_{i_0})\otimes e_{i_0})(x_{i_0})} =\abs{2f_{i_0}(x_{i_0})+h_{i_0}(x_{i_0})}\norm{e_{i_0}}>3-3\ep,\]
	and for $i \neq i_0$, 
	\[\norm{((2f_{i}+h_{i})\otimes e_{i})(x_{i_0})} = \abs{2f_{i}(x_{i_0})+h_{i}(x_{i_0})}\norm{e_i}<3\ep.\]
	Consequently, and since $\norm{T^{-1}}=1$, we have 
	\begin{equation}
	\nonumber
	\begin{aligned}
	&\norm{T(\sum_{i=1}^n (2f_{i}+h_{i})\otimes e_{i})} \geq \norm{\sum_{i=1}^n (2f_{i}+h_{i})\otimes e_{i}}
	\geq \norm{(\sum_{i=1}^n (2f_{i}+h_{i})\otimes e_{i})(x_{i_0})}
	\geq \\& \geq
	\norm{((2f_{i_0}+h_{i_0})\otimes e_{i_0})(x_{i_0})}-\sum_{i \neq i_0} \norm{((2f_{i}+h_{i})\otimes e_{i})(x_{i_0})} > 3-3\ep-(n-1)3\ep.
	\end{aligned}
	\end{equation}
	Hence there exists a point $y \in \B$ satisfying that
	\[\norm{T(\sum_{i=1}^n (2f_{i}+h_{i})\otimes e_{i})(y)}>3-3n\ep.\]
	Now, to finish the proof for $\alpha=0$ we need to show that for each $k=1, \ldots, p$,  $\norm{T(\sum_{i=1}^n g_i^k \otimes e_i)(y)} \geq \ep$. Thus fix $k \in \{1, \ldots, p\}$. Then for each $i=1, \ldots, n$, since $g_i^k<h_i+\ep$ and $f_i<g_i^k+\ep$, we have 
	\begin{equation}
	\nonumber
	\begin{aligned}
	-1-\ep \leq -g_i^k-\ep \leq 2f_i-g_i^k-\ep \leq 2f_i+h_i-2g_i^k \leq h_i+2\ep \leq 1+2\ep.
	\end{aligned}
	\end{equation} 
	Thus $\norm{(2f_i+h_i-2g_i^k) \otimes e_i} \leq 1+2\ep$.
	
	Moreover, for each $i=1, \ldots, n$ we have $g_i^k<h_i+\ep<2\ep$ on $\Ch_{\H_1} K_1 \setminus U_i$, and $f_i<\ep$ on $\Ch_{\H_1} K_1 \setminus V_i \supseteq \Ch_{\H_1} K_1 \setminus U_i$.
	Consequently, since the sets $U_i$ are pairwise disjoint, and for each $i=1, \ldots, n$ we have
	$\abs{2f_i+h_i-2g_i^k}<7\ep$ on $\Ch_{\H_1} K_1 \setminus U_i$, by Lemma \ref{omezenost} it follows that 
	\begin{equation}
	\nonumber
	\begin{aligned}
	&\norm{\sum_{i=1}^n (2f_{i}+h_{i}-2g_i^k)\otimes e_{i}} \leq  1+2\ep+7n\ep.
	\end{aligned}
	\end{equation}
	It follows that if $\norm{T(\sum_{i=1}^n g_i^k \otimes e_i)(y)}< \ep$, we would have
	\begin{equation}
	\nonumber
	\begin{aligned}
	&\norm{T}(1+2\ep+7n\ep) \geq \norm{T(\sum_{i=1}^n(2f_{i}+h_{i}-2g_i^k)\otimes e_{i})} 
	\geq \\& \geq
	\norm{T(\sum_{i=1}^n (2f_{i}+h_{i}-2g_i^k)\otimes e_{i})(y)} 
	\geq \\& \geq
	\norm{T(\sum_{i=1}^n (2f_{i}+h_{i})\otimes e_{i})(y)}-2\norm{T(\sum_{i=1}^n g_i^k\otimes e_{i})(y)} \geq 3-3n\ep-2\ep.
	\end{aligned}
	\end{equation}
	This contradicts the choice of $\ep$ and finishes the proof for $\alpha=0$.
	
	Now we assume that the statement b) holds for an ordinal $\alpha \geq 0$ and let $x_1, \ldots, x_n \in \W_{\H_1}^{(\alpha+1)}$ be points satisfying that $h_i(x_i)>1-\ep$ for $i=1, \ldots, n$. Then for $i=1, \ldots, n$, there exists a net $\{ x^i_{\lambda_i} \}_{\lambda_i \in \Lambda_i}$ of points in $\W_{\H_1}^{(\alpha)} \cap U_i$, distinct from $x_i$, converging to $x_i$, and satisfying that $h_i(x^i_{\lambda_i})>1-\ep$ for each $\lambda_i$. By the assumption, for each $(\lambda_1, \ldots, \lambda_n) \in \Lambda_1 \times \ldots \times \Lambda_n$ there exists a point 
	\[y_{(\lambda_i)_{i=1}^n} \in \Lambda_{(T, (h_i, e_i)_{i=1}^n, \ep)} ((x^i_{\lambda_i})_{i=1}^n) \cap \B^{(\alpha)}.\] Then we may consider $\{ y_{(\lambda_i)_{i=1}^n} \}$ as a net in $\B^{(\alpha)}$ with the ordering given by $(\lambda_i)_{i=1}^n\leq (\tilde{\lambda}_i)_{i=1}^n$ if $\lambda_i \leq \tilde{\lambda}_i$ for each $i=1, \ldots, n$. Since $\B^{(\alpha)}$ is compact, passing to a subnet we may assume that the net $\{ y_{(\lambda_i)_{i=1}^n} \}$ converges to a point $y \in \B^{(\alpha)}$. We want to show that $y \in \B^{(\alpha+1)}$. Assuming the contrary, the net $\{ y_{(\lambda_i)_{i=1}^n} \}$ is essentially constant, thus there exist $\lambda_1, \ldots, \lambda_n$ such that $y=y_{(\tilde{\lambda}_i)_{i=1}^n}$ for each $(\tilde{\lambda}_i)_{i=1}^n \geq (\lambda_i)_{i=1}^n$. But then for each $(\tilde{\lambda}_i)_{i=1}^n \geq (\lambda_i)_{i=1}^n$ we have 
	\[y=y_{(\tilde{\lambda}_i)_{i=1}^n} \in \Lambda_{(T, (h_i, e_i)_{i=1}^n, \ep)} ((x^i_{\tilde{\lambda}_i})_{i=1}^n),\]
	which contradicts a). Thus $y \in \B^{(\alpha+1)}$. 
	
	It remains to show that $y \in \Lambda_{(T, (h_i, e_i)_{i=1}^n, \ep)} ((x_i)_{i=1}^n)$. 
	Thus we choose arbitrary $(g_1, \ldots, g_n) \in \mathcal{F}((h_i, x_i)_{i=1}^n, \ep)$. Then for each $i=1, \ldots, n$, $g_i(x_i)>1-\ep$ by the definition, thus there exists $\lambda_i$ such that $g_i(x^i_{\tilde{\lambda}_i})>1-\ep$ for each $\tilde{\lambda}_i>\lambda_i$. This means that
	\[(g_i)_{i=1}^n \in \mathcal{F}((h_i, x^i_{\tilde{\lambda}_i})_{i=1}^n, \ep)\] for $(\tilde{\lambda}_i)_{i=1}^n \geq (\lambda_i)_{i=1}^n$. Thus for $(\tilde{\lambda}_i)_{i=1}^n \geq (\lambda_i)_{i=1}^n$ we have 
	\[\norm{T(\sum_{i=1}^n g_i \otimes e_i)(y_{(\tilde{\lambda}_i)_{i=1}^n})} \geq \ep.\] Consequently, $\norm{T(\sum_{i=1}^n g_i \otimes e_i)(y)} \geq \ep$ by continuity. This means that $y \in \Lambda_{(T, (h_i, e_i)_{i=1}^n, \ep)} ((x_i)_{i=1}^n)$.
	
	Finally, we assume that $\alpha$ is a limit ordinal and that b) holds for each ordinal $\beta<\alpha$. If $x_1, \ldots, x_n \in \W_{\H_1}^{(\alpha)}$, $h_i(x_i)>1-\ep$ for each $i$, then $x_1, \ldots, x_n \in \W_{\H_1}^{(\beta)}$ for each $\beta<\alpha$. Hence $\Lambda_{(T, (h_i, e_i)_{i=1}^n, \ep)} ((x_i)_{i=1}^n) \cap \B^{(\beta)} \neq \emptyset$ for each such $\beta$ by the assumption. Thus 
	\[\Lambda_{(T, (h_i, e_i)_{i=1}^n, \ep)} ((x_i)_{i=1}^n) \cap \B^{(\alpha)}=\bigcap_{\beta<\alpha} \Lambda_{(T, (h_i, e_i)_{i=1}^n, \ep)} ((x_i)_{i=1}^n) \cap \B^{(\beta)}\]
	is nonempty, as it is a nonincreasing intersection of compact sets. The proof is finished.
\end{proof}

Now we are ready to prove Theorem \ref{jednostranna}.

\emph{Proof of Theorem \ref{jednostranna}.}

\begin{proof}
	Without loss of generality we may assume that $\norm{T}<3$ and $\norm{T^{-1}}=1$, otherwise we would multiply $T$ by a suitable constant. Let $\A_1$ be the scalar function space associated to $\H_1$.
	
	a) Pick $x \in \W_{\H_1}^{(\alpha)}$ and 
	\[ 0<\ep<\frac{3-\norm{T}}{9\norm{T}+5},\]
	and let $e \in E_1$ be an arbitrary vector of norm $1$.  
	Since $x$ is a weak peak point, there exists a function $h \in \A_1^{E_1}$ satisfying that $0 \leq h \leq 1$ and $h(x)>1-\ep$. Then by Lemma \ref{main} b), the set 
	\[\Omega_{\H_2} \cap \B^{(\alpha)} \supseteq \Lambda_{(T, (h, e), \ep)}(x) \cap \B^{(\alpha)}\]
	is nonempty.
	
	b) Let $\alpha$ be an ordinal such that $\B^{(\alpha)} \cap \Omega_{\H_2}$ is finite. We assume that $\W_{\H_1}^{(\alpha)}$ is infinite and seek a contradiction. Then there exist pairwise disjoint open sets $\{U_m \}_{m \in \en}$ in $K_1$ and points $\{ x_m\}_{m \in \en} \subseteq \W_{\H_1}^{(\alpha)}$ such that $x_m \in U_m$ for each $m \in \en$. Fix $e \in S_{E_1}$, \[ 0<\ep<\frac{3-\norm{T}}{9\norm{T}+5},\]
	and find functions $\{g_m\}_{m \in \en}$ from $\A_1^{E_1}$ satisfying that for each $m \in \en$,
	\[0 \leq g_m \leq 1, g_m(x_m)>1-\frac{\ep}{2^m} \text{ and } g_m <\frac{\ep}{2^m} \text{ on } \Ch_{\H_1} K_1 \setminus U_m.\] 
	Then by Lemmas \ref{c_0} and \ref{omezenost}, for each $y \in K_2$ there are at most finitely many functions $g_m$ such that $\norm{T(g_m \otimes e)(y)} \geq \ep$. Thus there exists an index $m_0 \in \en$ satisfying that $\norm{T(g_{m_0} \otimes e)(y)}<\ep$ for each $y$ from the finite set $\B^{(\alpha)} \cap \Omega_{\H_2}$. This means that 
	\[\Lambda_{(T, (g_{m_0}, e), \ep)}(x_{m_0}) \cap \B^{(\alpha)} \cap \Omega_{\H_2}=\Lambda_{(T, (g_{m_0}, e), \ep)}(x_{m_0}) \cap \B^{(\alpha)}\]	is empty. But this contradicts Lemma \ref{main} b), which shows that $\W_{\H_1}^{(\alpha)}$ is finite.
	
	c) If $\alpha$ is an ordinal such that $\B^{(\alpha)} \cap \Omega_{\H_2}$ is finite, then we know by b) that $\W_{\H_1}^{(\alpha)}$ is also finite. Thus suppose that $\W_{\H_1}^{(\alpha)}=\{ x_1, \ldots, x_n\}$ and $\B^{(\alpha)} \cap \Omega_{\H_2}=\{y_1, \ldots, y_m\}$. We find pairwise disjoint open sets $U_1, \ldots, U_n$ such that $x_i \in U_i$ for each $i=1, \ldots, n$. 
	
	Fix 
	\[ 0<\ep<\frac{3-\norm{T}}{\norm{T}(7n+2)+3n+2},\]
	and let $h_1, \ldots, h_n$ be functions in $\A_1^{E_1}$ satisfying that 
	\[0 \leq h_i \leq 1, h_i(x_i)>1-\ep \text{ and } h_i<\ep \text{ on } \Ch_{\H_1} K_1 \setminus U_i.\]
	
	We define an operator $S:(E_1^n, \norm{\cdot}_{\max}) \rightarrow \C(\B^{(\alpha)} \cap \Omega_{\H_2}, E_2) \simeq (E_2^m, \norm{\cdot}_{\max})$ as
	\[S(e_1, \ldots, e_n)=T(\sum_{i=1}^n h_i \otimes e_i)|_{(\B^{(\alpha)} \cap \Omega_{\H_2})}.\] 
	Then $S$ is clearly linear. 
	
	Moreover, let $e_1, \ldots, e_n \in E_1$ be nonzero vectors with $\max_{i=1, \ldots , n} \norm{e_i}=1$. Then 
	\begin{equation}
	\nonumber
	\begin{aligned}
	&\norm{S(e_1, \ldots, e_m)}_{\max} \leq \norm{T(\sum_{i=1}^n h_i \otimes e_i)}\leq  \norm{T} \norm{\sum_{i=1}^n h_i \otimes e_i}\leq \norm{T}(1+n\ep)
	\end{aligned}
	\end{equation}
by Lemma \ref{omezenost}. At the same time, by Lemma \ref{main} b), there exists a point $y \in \B^{(\alpha)}$ such that $\norm{T(\sum_{i=1}^n h_i \otimes e_i)(y)} \geq \ep$. It follows that $\ep\leq \norm{S} \leq \norm{T}(1+n\ep)$, and hence $S$ is an isomorphism from $E_1^n$ into $E_2^m$.

	d) If $\H_1$ contains the constant functions, then so does the space $\A_1^{E_1}$. Let $h \in \A_1^{E_1}$ be the constant function $1$, 	
\[0<\ep<\frac{3-\norm{T}}{9\norm{T}+5},\]
and let $e \in S_{E_1}$. Then again by Lemma \ref{main} b), for each $x_0 \in \W_{\H_1}^{(\alpha)}$ there exists $y \in \Lambda_{(T, (h, e), \ep)}(x_0) \cap \B^{(\alpha)} \subseteq \Omega_{\H_2} \cap \B^{(\alpha)}$. Moreover, by Lemma \ref{main} a), each $y \in \B^{(\alpha)}$ belongs to at most finitely many sets $\Lambda_{(T, (h, e), \ep)}(x)$, where $x \in \W_{\H_1}^{(\alpha)}$.  Thus if $\B^{(\alpha)} \cap \Omega_{\H_2}$ is infinite, then 
\[\abs{\W_{\H_1}^{(\alpha)}} \leq \abs{\omega}\abs{\B^{(\alpha)} \cap \Omega_{\H_2}}=\abs{\B^{(\alpha)} \cap \Omega_{\H_2}}.\]
\end{proof}

To finish this section we now show that we cannot omit the assumption that the boundaries are closed in Theorem \ref{homeomorfni}.

\begin{example}
	\label{Priklad}
	Let $K_1=[1, \omega 2]$, $K_2=[1, \omega]$ and
	\[ \H_1=\{ f \in \C([1, \omega 2]): f(\omega 2)=0\}, \H_2=\C([1, \omega]).\]
	Then each point of the countable Choquet boundaries of $\H_1$ and $\H_2$ is a weak peak point, but $d_{BM}(\H_1, \H_2)=2$.
\end{example}

\begin{proof}
	It is clear that $\W_{\H_1}=[1, \omega 2)$. Thus $\Ch_{\H_1} K_1 \supseteq \W_{\H_1}=[1, \omega 2)$. Since zero is not an extreme point of $B_{\H_1^*}$, we have $\Ch_{\H_1} K_1=\W_{\H_1}=[1, \omega 2)$. Moreover, it is clear that $\W_{\H_2}=\Ch_{\H_2} K_2=[1, \omega]$. Thus $d_{BM}(\H_1, \H_2) \geq 2$ by \cite[Theorem 1.1]{rondos-spurny-spaces}. 
	
	On the other hand, we define a mapping $T: \H_1 \rightarrow \H_2$ by
	\begin{equation}
	\nonumber
	\begin{aligned}
	& Tf(1)=f(1) \\&
	Tf(2m)=f(m)+f(\omega+m), \quad m \in \en, \\&
	Tf(2m+1)=f(m)-f(\omega+m), \quad m \in \en, \\&
	Tf(\omega)=f(\omega).
	\end{aligned}
	\end{equation}
	Then it is easy to check that $T$ is an isomorphism from $\H_1$ onto $\H_2$ with $S=T^{-1}$ given by
	\begin{equation}
	\nonumber
	\begin{aligned}
	& Sf(1)=f(1) \\&
	Sf(m)=\frac{f(2m)+f(2m+1)}{2}, \quad m \in \en, \\&
	Sf(\omega)=f(\omega), \\&
	Sf(\omega+m)=\frac{f(2m)-f(2m+1)}{2}, \quad m \in \en.
	\end{aligned}
	\end{equation}
	Moreover, it is easy to see that $\norm{T}=2$ and $\norm{S}=1$, which finishes the proof.
\end{proof}

\section{Isomorphisms between spaces with boundaries of finite height}

In this section, for $i=1, 2$, let $K_i$ be a compact space, $E_i$ be a Banach space and $\H_i$ be a closed subspace of $\C(K_i, E_i)$. Let $E_2$ do not contain an isomorphic copy of $c_0$, let $\A_1$ be the scalar function space associated to $\H_1$, $\B$ be a closed boundary for $\H_2$ and $T:\H_1 \rightarrow \H_2$ be an into isomorphism. In this setting, we prove three auxiliary results, which we then use for the proof of Theorem \ref{distance}. 

\begin{lemma}
	\label{finite}
	Let $0 <\ep <1$, $\zeta>0$, $e$ be a nonzero vector in $E_1$,  $\{U_n\}$ be a sequence of pairwise disjoint open subsets of $K_1$ and let the sequence of functions $\{ h_n\}$ in $B_{\A_1}$ satisfies that for each $n \in \en$, $h_n \otimes e \in \H_1$ and $\abs{h_n}<\frac{\ep}{2^n}$ on $\Ch_{\H_1} K_1 \setminus U_n$. Let $G$ be a finite subset of $K_2$. Then there exists an $m \in \en$ such that $\R(T(h_m \otimes e), \zeta) \cap G=\emptyset$. 
\end{lemma}

\begin{proof}
	We assume that this is not true and seek a contradiction. Let $G=\{ y_1, \ldots , y_k\}$, and we denote 
	\[\Lambda_i = \{ n \in \en: \norm{T(h_n \otimes e) (y_i)} \geq \zeta \}, \quad i=1, \ldots, k.\]
	By the assumption, the union $\bigcup_{i=1}^k \Lambda_i$ contains all natural numbers. Thus we can fix an index $i_0$ such that the set $\Lambda_{i_0}$ is infinite. Since the sets $U_n$ are pairwise disjoint, by Lemma \ref{omezenost} for each $n \in \en$ and $\alpha_1, \ldots , \alpha_n \in S_{\er}$ it holds that
	\[\norm{\sum_{i=1}^n \alpha_i h_i(x)} \leq 1+\ep.\]
	But since $E_2$ does not contain an isomorphic copy of $c_0$, the sequence $\{ h_n\}_{n \in \Lambda_{i_0}}$ contradicts Lemma \ref{c_0}. The proof is finished.
\end{proof}

The following Lemma, that is motivated by \cite[Proposition 2.3]{CANDIDOc0}, is a key ingredient, which makes use of the assumption that the considered heights are finite. For the proof, notice that if $F \subseteq K_2$ is a compact set, and the set $F \cap \B^{(k)}$ is infinite for some $k \in \en$, then $F \cap \B^{(k+1)}$ is nonempty. Indeed, since $F \cap \B^{(k)}$ is an infinite compact set, it has a cluster point, and this point belongs to $\B^{(k+1)}$.

\begin{lemma}
\label{system}
Let $n, k, l \in \en$, $k \leq l$, $\B^{(k)} \cap \Omega_{\H_2}$ is finite, and let $x_1, \ldots, x_n$ be distinct points in $\W_{\H_1}^{(l)}$.
Let $e_1, \ldots, e_n$ be nonzero vectors in $E_1$ with $\max_{j=1, \ldots , n} \norm{e_i}=1$, and $U_1, \ldots, U_n$ be pairwise disjoint open sets such that $x_j \in U_j$ for each $j=1, \ldots, n$. 
Suppose that for a given $0<\ep<1$, there exist functions $h_1, \ldots, h_n \in \A_1^{E_1}$ satisfying that for each $j=1, \ldots, n$, 
\[0 \leq h_j \leq 1, h_j(x_j)>1-\ep, h_j<\ep \text{ on } \Ch_{\H_1} K_1 \setminus U_j,\] 
and such that the set $\R(\sum_{j=1}^n T(h_j \otimes e_j), \ep) \cap \B^{(k)}$ is empty. 

Then there exist functions $(h_1^i)_{i=0}^{l}, \ldots , (h_n^i)_{i=0}^{l}$ in $\A_1^{E_1}$ satisfying that for each $j=1, \ldots, n$ and $i=0, \ldots, l$,  
\[0 \leq h_j^i \leq 1, h_j^i<\ep \text{ on } \Ch_{\H_1} K_1 \setminus U_j, h_j^i<h_j^{i+1}+\ep \text{ for } i<l,\] 
$h_j^i(x_j^i) \geq 1-\ep$ for some $x_j^i \in \W_{\H_1}^{(i)}$ and such that for every $y \in \B$, the set \[\{ i \in \{0, \ldots , l\}: y \in \R(\sum_{j=1}^n T(h_j^i \otimes e_j), \ep)\}\]
 has cardinality at most $k$.
\end{lemma}

\begin{proof}
The Lemma will be proved once we verify that the following Claim is true. 
		
\emph{Claim.
For each $i=l, l-1, \ldots, 0$ there exist functions
\[\{h_1^l, \ldots, h_n^l\}, \{h_1^{l-1}, \ldots, h_n^{l-1}\} \ldots , \{h_1^i, \ldots, h_n^i\}\] in $\A_1^{E_1}$ and points 
\[\{x_1^l, \ldots, x_n^l\}, \{x_1^{l-1}, \ldots, x_n^{l-1}\}, \ldots, \{x_1^i, \ldots, x_n^i\}\] such that for each $j=1, \ldots, n$ and $m=i, \ldots, l$, $x_j^m \in \W_{\H_1}^{(m)}$,
\[0 \leq h_j^m \leq 1, h_j^m<\ep \text{ on } \Ch_{\H_1} K_1 \setminus U_j,
h_j^m(x_j^m) \geq 1-\ep, h_j^m<h_j^{m+1}+\ep \text{ for } m<l\] and such that
\begin{equation}
\label{empty}
 \bigcup_{s=0}^{l-i} \{\B^{(\max\{k-s, 0\})} \cap \bigcup_{i \leq r_1< \ldots <r_{s+1} \leq l} \bigcap_{p \in \{r_1, \ldots , r_{s+1}\}} \R(\sum_{j=1}^n T(h_j^p   \otimes e_j), \ep)\}=\emptyset. 
\end{equation}}

Indeed, assume first that the Claim holds. This in particular for $i=0$ and $s=k$ means that 
\[\emptyset=\B \cap \bigcup_{0 \leq r_1< \ldots <r_{k+1} \leq l} \bigcap_{p \in \{r_1, \ldots , r_{k+1}\}} \R(\sum_{j=1}^n T(h_j^p   \otimes e_j), \ep).\]
This means that for each $0 \leq r_1 < \ldots <r_{k+1} \leq l$, the intersection 
\[\B \cap \bigcap_{p \in \{r_1, \ldots , r_{k+1}\}} \R(\sum_{j=1}^n T(h_j^p \otimes e_j), \ep)\] is empty, which we wanted to prove. 

\emph{Proof of Claim.}
We proceed by reverse induction for $i=l, l-1, \ldots, 0$. For the case $i=l$, it is enough to denote $h_j^l=h_j$ and $x_j^l=x_j$ for $j=1, \ldots, n$, since then \eqref{empty} holds due to the assumption that
$\R(\sum_{j=1}^n T(h_j \otimes e_j), \ep) \cap \B^{(k)}$ is empty.

 Thus suppose that $0 \leq i<l$, and that we have found the functions \[\{h_1^l, \ldots, h_n^l\}, \{h_1^{l-1}, \ldots, h_n^{l-1}\} \ldots , \{h_1^{i+1}, \ldots, h_n^{i+1}\}\] in $\A_1^{E_1}$ and points 
 \[\{x_1^l, \ldots, x_n^l\}, \{x_1^{l-1}, \ldots, x_n^{l-1}\}, \ldots, \{x_1^{i+1}, \ldots, x_n^{i+1}\}\] 
 such that \eqref{empty} holds for $i+1$, and satisfying all the other above conditions. 
Hence we know that the set 

\[ \bigcup_{s=0}^{l-i-1} \{\B^{(\max\{k-s, 0\})} \cap \bigcup_{i+1 \leq r_1< \ldots <r_{s+1} \leq l} \bigcap_{p \in \{r_1, \ldots , r_{s+1}\}} \R(\sum_{j=1}^n T(h_j^p \otimes e_j), \ep)\}\]
is empty. Since for each $s=0, \ldots, l-i-1$, the set \[\bigcup_{i+1 \leq r_1< \ldots <r_{s+1} \leq l} \bigcap_{p \in \{r_1, \ldots , r_{s+1}\}} \R(\sum_{j=1}^n T(h_j^p \otimes e_j), \ep)\] is compact, the set
\[ \B^{(\max\{k-s-1, 0\})} \cap \bigcup_{i+1 \leq r_1< \ldots <r_{s+1} \leq l} \bigcap_{p \in \{r_1, \ldots , r_{s+1}\}} \R(\sum_{j=1}^n T(h_j^p \otimes e_j), \ep)\]
is finite (see the note above the statement of the lemma). 

The set $\B^{(k)}  \cap \Omega_{\H_2}$ is also finite by the assumption, hence so is the set
\begin{equation}
\nonumber
\begin{aligned}
G=&(\B^{(k)}  \cap \Omega_{\H_2})
\cup \\& 
\bigcup_{s=0}^{l-i-1} \{\B^{(\max\{k-s-1, 0\})} \cap \bigcup_{i+1 \leq r_1< \ldots <r_{s+1} \leq l} \bigcap_{p \in \{r_1, \ldots , r_{s+1}\}} \R(\sum_{j=1}^n T(h_j^p \otimes e_j), \ep)\}.
\end{aligned}
\end{equation} 

Next, for each $j=1, \ldots n$, by the continuity of the function $h_j^{i+1}$ we find an open set $V \subseteq U_j$ containing $x_j^{i+1}$, and such that 
$h_j^{i+1}>1-\ep$ on $V$. Next, since $x_j^{i+1} \in \W_{\H_1}^{(i+1)}$, we may find a sequence of points $\{z_j^m\}_{m \in \en} \subseteq \W_{\H_1}^{(i)}$ distinct from $x_j^{i+1}$, together with their pairwise disjoint open neighborhoods $\{U_j^{m}\}_{m \in \en}$, all contained in $V$. Since each point $z_j^m$ is a weak peak point, we may find a function $g_j^m \in \A_1^{E_1}$ satisfying that
\[ 0 \leq g_j^m \leq 1, g_j^m(z_j^m)>1-\frac{\ep}{2^m} \text{ and } g_j^m<\frac{\ep}{2^m} \text{ on } \Ch_{\H_1} K_1 \setminus U_j^m.\]
Then for each $m \in \en$, $g_j^m \leq h_j^{i+1}+\ep$ by Lemma \ref{nerovnost}. Using Lemma \ref{finite}, for each $j=1, \ldots, n$ we find an index $m_j$ such that the function $g_j^{m_j}$ satisfies that $\R(T(g_j^{m_j} \otimes e_j, \frac{\ep}{n}) \cap G$ is empty, and we denote $h_j^i=g_j^{m_j}$ and $x_j^i=z_j^{m_j}$. Thus for $y \in G$ it holds that
\[\norm{\sum_{j=1}^n T(h_j^i \otimes e_j)(y)} \leq \sum_{j=1}^n \norm{T(h_j^i \otimes e_j)(y)} < n\frac{\ep}{n}=\ep,\]
and hence $\R(\sum_{j=1}^n T(h_j^i \otimes e_j), \ep) \cap G$ is empty. Consequently, 
\begin{equation}
\nonumber
\begin{aligned}
\emptyset&=\R(\sum_{j=1}^n T(h_j^i \otimes e_j), \ep) \cap G=
\R(\sum_{j=1}^n T(h_j^i \otimes e_j), \ep) \cap  \Big((\B^{(k)}  \cap \Omega_{\H_2}) \cup \\& \bigcup_{s=0}^{l-i-1} \{\B^{(\max\{k-s-1, 0\})} 
\cap \bigcup_{i+1 \leq r_1< \ldots <r_{s+1} \leq l} \bigcap_{p \in \{r_1, \ldots , r_{s+1}\}} \R(\sum_{j=1}^n T(h_j^p \otimes e_j), \ep)\}\Big)
=\\&=
(\R(\sum_{j=1}^n T(h_j^i \otimes e_j), \ep) \cap \B^{(k)}) \cup  \bigcup_{s=0}^{l-i-1} \{\B^{(\max\{k-s-1, 0\})} 
\cap \\& \bigcup_{i+1 \leq r_1< \ldots <r_{s+1} \leq l} \bigcap_{p \in \{r_1, \ldots , r_{s+1}\}} \R(\sum_{j=1}^n T(h_j^p \otimes e_j), \ep) \cap \R(\sum_{j=1}^n T(h_j^i \otimes e_j), \ep)\}
=\\&=
(\R(\sum_{j=1}^n T(h_j^i \otimes e_j), \ep) \cap \B^{(k)}) \cup \\&  \bigcup_{s=0}^{l-i-1} \{\B^{(\max\{k-s-1, 0\})} 
\cap \bigcup_{i=r_1< \ldots <r_{s+2} \leq l} \bigcap_{p \in \{r_1, \ldots , r_{s+2}\}} \R(\sum_{j=1}^n T(h_j^p \otimes e_j), \ep)\}
=\\&=
\bigcup_{s=-1}^{l-i-1} \{\B^{(\max\{k-s-1, 0\})} 
\cap \bigcup_{i=r_1< \ldots <r_{s+2} \leq l} \bigcap_{p \in \{r_1, \ldots , r_{s+2}\}} \R(\sum_{j=1}^n T(h_j^p \otimes e_j), \ep)\}
=\\&=
\bigcup_{s=0}^{l-i} \{\B^{(\max\{k-s, 0\})} \cap \bigcup_{i=r_1< \ldots <r_{s+1} \leq l} \bigcap_{p \in \{r_1, \ldots , r_{s+1}\}} \R(\sum_{j=1}^n T(h_j^p \otimes e_j), \ep)\}.
\end{aligned}
\end{equation}
Thus recalling the inductive assumption we conclude that the set
\begin{equation}
\nonumber
\begin{aligned}
&\bigcup_{s=0}^{l-i} \{\B^{(\max\{k-s, 0\})} \cap \bigcup_{i \leq r_1< \ldots <r_{s+1} \leq l} \bigcap_{p \in \{r_1, \ldots , r_{s+1}\}} \R(\sum_{j=1}^n T(h_j^p \otimes e_j), \ep)\}=\\&
\bigcup_{s=0}^{l-i} \{\B^{(\max\{k-s, 0\})} \cap \bigcup_{i=r_1< \ldots <r_{s+1} \leq l} \bigcap_{p \in \{r_1, \ldots , r_{s+1}\}} \R(\sum_{j=1}^n T(h_j^p \otimes e_j), \ep)\}
\cup \\& \cup \bigcup_{s=0}^{l-i-1} \{\B^{(\max\{k-s, 0\})} \cap \bigcup_{i+1 \leq r_1< \ldots <r_{s+1} \leq l} \bigcap_{p \in \{r_1, \ldots , r_{s+1}\}} \R(\sum_{j=1}^n T(h_j^p \otimes e_j), \ep)\}
\end{aligned}
\end{equation}
is empty. This finishes the induction step and the proof. 
\end{proof}

Next, similarly as in section 3 we define certain sets $\G$, $\Theta$, which we use for the proof of Theorem \ref{distance}. For $s, m \in \en$, $m < s$, $\ep>0$ and functions $h^1, \ldots , h^s \in \A_1^{E_1}$ we denote
\begin{equation}
\nonumber
\begin{aligned}
 &\mathcal{G}(h^1, \ldots, h^s, m, \ep)=\{ (g^1, \ldots, g^m) \in (\A_1^{E_1})^m, \\& \forall i=1, \ldots, m: 0 \leq g^i \leq 1, g^i \geq h^i-\ep \text{ and } g^i \leq h^{s-m+i}+\ep\}. 
\end{aligned}
\end{equation}

Next, for $m, n, s \in \en$, $m < s$, $\ep>0$, vectors $e_1, \ldots, e_n \in E_1$ and functions \\$(h_1^r)_{r=1}^s, \ldots , (h_n^r)_{r=1}^s$ from $\A_1^{E_1}$ we denote
\begin{equation}
\nonumber
\begin{aligned}
& \Theta(T, (h_1^r)_{r=1}^s, \ldots , (h_n^r)_{r=1}^s, (e_j)_{j=1}^n, m, \ep) = \\& \bigcap \{ \R(T(\sum_{j=1}^n \sum_{i=1}^m g_{j}^i \otimes e_j), \ep): \forall j=1, \ldots n, (g_j^1, \ldots, g_j^m) \in \mathcal{G}(h_j^1, \ldots, h_j^s, m, \ep) \}.
\end{aligned}
\end{equation}

Suppose that the functions $(h_1^r)_{r=1}^s, \ldots , (h_n^r)_{r=1}^s$ from $\A_1^{E_1}$ satisfy that for each $j=1, \ldots, n$ and $r=1, \ldots, s$, $0 \leq h_j^r \leq 1$ and for $r<s$, $h_j^r<h_j^{r+1}+\ep$. Then for $\tilde{\ep} \geq (s-1)\ep$ it holds that $h_j^r \leq h_j^t+\tilde{\ep}$ for each $r, t \in \{1, \ldots, s\}, r<t$. Consequently, for each $j=1, \ldots, n$ and $m \leq s$, the set $\mathcal{G}(h_j^1, \ldots, h_j^s,m, \tilde{\ep})$ is nonempty. Indeed, for each $m$ indices $r_1, \ldots, r_m \in \{1, \ldots, s\}$, 
\begin{equation}
\label{nalezi}
(h_j^{r_1}, \ldots, h_j^{r_m}) \in \mathcal{G}(h_j^1, \ldots, h_j^s,m, \tilde{\ep}).
\end{equation} 
Thus the set of functions defining 
\[\Theta(T, (h_1^r)_{r=1}^s, \ldots , (h_n^r)_{r=1}^s, (e_j)_{j=1}^n, m, \tilde{\ep})\] is also nonempty. 

The following Lemma, which is an another crucial ingredient for the proof of Theorem \ref{distance}, is inspired by \cite[Lemma 2.1]{CANDIDOc0}.

\begin{lemma}
\label{nonempty}
Let $n, m, s \in \en$, $m<s$, $\norm{T^{-1}}=1$, $\norm{T}<\frac{s+m}{s-m}$. Choose $\ep \in (0, \frac{1}{n})$ satisfying
\[\ep<\frac{s+m-\norm{T}(s-m)}{\norm{T}(2m+n(s+5m))+n(s+m)+2}\]
and let nonzero vectors $e_1, \ldots, e_n \in E_1$ satisfy that $\max_{i=1, \ldots , n} \norm{e_i}=1$. Let $U_1, \ldots, U_n$ be pairwise disjoint open sets in $K_1$, and functions $(h_1^r)_{r=1}^s, \ldots , (h_n^r)_{r=1}^s$ from $\A_1^{E_1}$ satisfy for each $r=1, \ldots, s$ and $j=1, \ldots, n$,
\[0 \leq h_j^r \leq 1, h_j^r<\ep \text{ on } \Ch_{\H_1} K_1 \setminus U_j,  
\text{ and for } r<s, h_j^r \leq h_j^{r+1} +\frac{\ep}{s}.\]
Choose $j_0 \in \{1, \ldots, n\}$ satisfying that $\norm{e_{j_0}}=1$. If $\norm{h_{j_0}^1} > 1-\frac{\ep}{s}$, then 
\[\Theta(T, (h_1^r)_{r=1}^s, \ldots , (h_n^r)_{r=1}^s, (e_j)_{j=1}^n, m, \ep) \cap \B\] is nonempty.
\end{lemma}

\begin{proof}
By the assumptions, there exists a point $x_0 \in \Ch_{\H_1} K_1$ satisfying that
\[h_{j_0}^r(x_0) \geq 1-(s-1)\frac{\ep}{s}>1-\ep, \quad r=1, \ldots, s. \]
For $j \neq j_0$ and $r=1, \ldots, s$ we have $h_{j}^r(x_0)<\ep$ by disjointness of the sets $U_1, \ldots , U_n$. Consequently, since $\norm{T^{-1}}=1$, we have
\begin{equation}
\nonumber
\begin{aligned}
&\norm{T \Big(\sum_{j=1}^n(2\sum_{r=1}^m h_{j}^r+\sum_{r=m+1}^s h_j^r) \otimes e_j\Big)} \geq
\norm{\sum_{j=1}^n(2\sum_{r=1}^m h_{j}^r+\sum_{r=m+1}^s h_j^r) \otimes e_j)} \geq \\& \geq 
\norm{(\sum_{j=1}^n(2\sum_{r=1}^m h_{j}^r+\sum_{r=m+1}^s h_j^r) \otimes e_j)(x_0)} \geq
\norm{((2\sum_{r=1}^m h_{j_0}^r+\sum_{r=m+1}^s h_{j_0}^r) \otimes e_{j_0})(x_0)}-\\&\norm{(\sum_{j \neq j_0}(2\sum_{r=1}^m h_{j}^r+\sum_{r=m+1}^s h_j^r) \otimes e_j)(x_0)}>\\&> (2m+s-m)(1-\ep)-(n-1)(2m+s-m)\ep=(s+m)(1-n\ep).
\end{aligned}
\end{equation} 

Thus there exists a point $y \in \B$ such that
\[ \norm{T(\sum_{j=1}^n(2\sum_{r=1}^m h_{j}^r+\sum_{r=m+1}^s h_j^r) \otimes e_j)(y)} > (s+m)(1-n\ep).\]

We claim that $y \in \Theta(T, (h_1^r)_{r=1}^s, \ldots , (h_n^r)_{r=1}^s, (e_j)_{j=1}^n, m, \ep)$. Thus we choose arbitrary functions $(g_1^i)_{i=1}^m, \ldots, (g_n^i)_{i=1}^m$ in $\A_1^{E_1}$ satisfying that for each $j=1, \ldots n$, $(g_j^1, \ldots, g_j^m) \in \mathcal{G}(h_j^1, \ldots, h_j^s, m, \ep)$ (we know from \eqref{nalezi} that set of such functions is nonempty). Fix $j \in \{1, \ldots, n\}$. By definition of the set $\mathcal{G}(h_j^1, \ldots, h_j^s, m, \ep)$ we have 
\begin{equation}
\label{prvni}
(\sum_{i=1}^m h_j^i)-m\ep \leq \sum_{i=1}^m g_j^i
\end{equation}
and
\begin{equation}
\label{druha}
\sum_{i=1}^m g_j^i \leq (\sum_{i=s-m+1}^s h_j^i)+m\ep.
\end{equation}
Consequently,
\begin{equation}
\nonumber
\begin{aligned}
-(s-m)-2m\ep&\leq -(\sum_{i=m+1}^s h_j^i)-2m\ep 
\leq \\& \leq  (2\sum_{i=1}^{s-m} h_j^i-\sum_{i=m+1}^s h_j^i)-2m\ep \leq^{\eqref{druha}} \\& \leq
2\sum_{i=1}^{s-m} h_j^i+2\sum_{i=s-m+1}^s h_j^i-2\sum_{i=1}^m g_j^i-\sum_{i=m+1}^s h_j^i=\\&
=2\sum_{i=1}^s h_j^i-2\sum_{i=1}^m g_j^i-\sum_{i=m+1}^s h_j^i= \\&=
2\sum_{i=1}^m h_j^i+\sum_{i=m+1}^s h_j^i-2\sum_{i=1}^m g_j^i
\leq^{\eqref{prvni}} \\& \leq (\sum_{i=m+1}^{s} h_j^i)+2m\ep \leq s-m+2m\ep.
\end{aligned}
\end{equation}	
Thus \[\norm{(2\sum_{i=1}^m h_j^i+\sum_{i=m+1}^s h_j^i-2\sum_{i=1}^m g_j^i) \otimes e_j} \leq s-m+2m\ep.\]
Next, notice that for $j \in \{1, \ldots, n\}$, $i \in \{1, \ldots, m\}$ and $x \in \Ch_{\H_1} K_1 \setminus U_j$ we have
$g_j^i(x) \leq h_j^{s-m+i}(x) +\ep < 2\ep$, hence \[\norm{((2\sum_{i=1}^m h_j^i+\sum_{i=m+1}^s h_j^i-2\sum_{i=1}^m g_j^i) \otimes e_j)(x)}<2m\ep+(s-m)\ep+4m\ep=(s+5m)\ep.\]

Thus by Lemma \ref{omezenost} we have 
\begin{equation}
\nonumber
\begin{aligned}
&\norm{\sum_{j=1}^n(2\sum_{i=1}^m h_j^i+\sum_{i=m+1}^s h_j^i-2\sum_{i=1}^m g_j^i) \otimes e_j}
\leq s-m+2m\ep+n(s+5m)\ep.
\end{aligned}
\end{equation}

Now we claim that $\norm{T(\sum_{j=1}^n \sum_{i=1}^m g_{j}^i \otimes e_j)(y)} \geq \ep$. Indeed, otherwise we would have
\begin{equation}
\nonumber
\begin{aligned}
&\norm{T}(s-m+2m\ep+n(s+5m)\ep)
\geq \\& \geq
\norm{T(\sum_{j=1}^n(2\sum_{i=1}^m h_j^i+\sum_{i=m+1}^s h_j^i-2\sum_{i=1}^m g_j^i) \otimes e_j)}
\geq \\& \geq
\norm{T(\sum_{j=1}^n(2\sum_{i=1}^m h_j^i+\sum_{i=m+1}^s h_j^i-2\sum_{i=1}^m g_j^i) \otimes e_j)(y)}
\geq \\& \geq 
\norm{T(\sum_{j=1}^n(2\sum_{i=1}^m h_j^i+\sum_{i=m+1}^s h_j^i) \otimes e_j)(y)}- 2\norm{T(\sum_{j=1}^n\sum_{i=1}^m g_j^i \otimes e_j))(y)}
 >\\&>
(s+m)(1-n\ep)-2\ep,
\end{aligned}
\end{equation}
contradicting the choice of $\epsilon$. Thus $y \in \Theta(T, (h_1^r)_{r=1}^s, \ldots , (h_n^r)_{r=1}^s, (e_j)_{j=1}^n, m, \ep)$, and we are done.
\end{proof}

Now we are ready to prove Theorem \ref{distance}.

\emph{Proof of Theorem \ref{distance}.}

\begin{proof}
We only need to prove that $\norm{T}\norm{T^{-1}} \geq \frac{2l+2-k}{k}$, since the bound $3$ follows from statements a), b) and c) of Theorem \ref{jednostranna}. Let $\A_1$ be the canonical scalar function space associated to $\H_1$.

\emph{Claim. There exists $n \in \en$ such that for each $\zeta>0$, there exist distinct points $(x_j)_{j=1}^n \in \W_{\H_1}^{(l)}$,
vectors $e_1, \ldots, e_n$ in $E_1$ with $\max_{j=1, \ldots , n} \norm{e_i}=1$, $U_1, \ldots, U_n$ pairwise disjoint open sets such that $x_j \in U_j$ for each $j=1, \ldots, n$, and functions $h_1, \ldots, h_n \in \A_1^{E_1}$ satisfying that for each $j=1, \ldots, n$, 
\[0 \leq h_j \leq 1, h_j(x_j)>1-\zeta, h_j<\zeta \text{ on } \Ch_{\H_1} K_1 \setminus U_j\] 
and such that the set $\R(\sum_{j=1}^n T(h_j \otimes e_j), \zeta) \cap \B^{(k)}$ is empty.}

Suppose first that the Claim holds, and assume for a contradiction that the isomorphism $T$ satisfies that $\norm{T^{-1}}=1$ and $\norm{T}<\frac{2l+2-k}{k}$. We find an $n \in \en$ from the Claim, and we choose $\ep \in (0, \frac{1}{n})$ satisfying

\[\ep<\frac{2l+2-k-\norm{T}k}{\norm{T}(2l+2-2k+n(6l+6-5k))+n(2l+2-k)+2} \text{, and let } \zeta = \frac{\ep}{l+1}.\]

Then we use the Claim to find all the above objects for this $\zeta$. Next, since in all cases (i), (ii) and (iii), the set $\B^{(k)} \cap \Omega_{\H_2}$ is finite, we can use Lemma \ref{system} to find functions $(h_1^r)_{r=0}^{l}, \ldots , (h_n^r)_{r=0}^{l}\in \A_1^{E_1}$ satisfying that for each $j=1, \ldots, n$ and $i=0, \ldots, l$,  
\[0 \leq h_j^i \leq 1, h_j^i<\zeta \text{ on } \Ch_{\H_1} K_1 \setminus U_j^i, h_j^i<h_j^{i+1}+\zeta \text{ for } i<l,\] 
$h_j^i(x_j^i) \geq 1-\zeta$ for some $x_j^i \in \W_{\H_1}^{(i)}$ and such that for every $y \in \B$,
\begin{equation}
\label{maxk}
\abs{\{ i \in \{0, \ldots , l\}: y \in \R(\sum_{j=1}^n T(h_j^i \otimes e_j), \zeta)\}} \leq k.
\end{equation}
Next, since $\zeta=\frac{\ep}{l+1}$, we can use Lemma \ref{nonempty} for $s=l+1$ and $m=l-k+1$ to obtain a point 
\[y \in \Theta(T, (h_1^r)_{r=0}^{l}, \ldots , (h_n^r)_{r=0}^{l}, (e_j)_{j=1}^n, l-k+1, \ep) \cap \B.\]
 By \eqref{maxk}, there are at most $k$ indices $i$ from the set $\{0, \ldots , l\}$ satisfying that $\norm{\sum_{j=1}^n T(h_j^i \otimes e_j)(y)} \geq \zeta$. Thus we can fix $i_1, \ldots, i_{l-k+1} \in \{0, \ldots, l\}$ such that for each $p=1, \ldots , l-k+1$ it holds that $\norm{\sum_{j=1}^n T(h_j^{i_p} \otimes e_j)(y)}<\zeta$. But then, since $\zeta=\frac{\ep}{l+1}$, 
 \[(h_j^{i_1}, \ldots, h_j^{i_{l-k+1}}) \in \mathcal{G}(h_j^0, \ldots, h_j^l, l-k+1, \ep)\]
 for each $j=1, \ldots, n$, see \eqref{nalezi}. 
 
 Thus by the definition of $\Theta(T, (h_1^r)_{r=0}^{l}, \ldots , (h_n^r)_{r=0}^{l}, (e_j)_{j=1}^n, m, \ep)$ we know that \[\norm{T(\sum_{j=1}^n\sum_{p=1}^{l-k+1} h_j^{i_p} \otimes e_j)(y)} \geq \ep.\]
  Consequently, we have
\[\ep \leq \norm{T(\sum_{j=1}^n\sum_{p=1}^{l-k+1} h_j^{i_p} \otimes e_j)(y)} \leq \sum_{p=1}^{l-k+1} \norm{T(\sum_{j=1}^n h_j^{i_p} \otimes e_j)(y)} < (l-k+1)\zeta<\ep.\]  
This contradiction shows that $\norm{T}\norm{T^{-1}} \geq \frac{2l+2-k}{k}$. Thus to finish the proof it is enough to prove the above Claim.

\emph{Proof of Claim.}\\
In case (i), we put $n=1$. Since $\W_{\H_1}^{(l)}$ is nonempty, we can take an arbitrary point $x_1 \in \W_{\H_1}^{(l)}$ and vector $e_1 \in S_{E_1}$, and then, since the set $\B^{(k)} \cap \Omega_{\H_2}$ is empty, the function $h_1$ exists simply by the definition of a weak peak point. 

For the case (ii), we again put $n=1$, and we pick an arbitrary $e_1 \in S_{E_1}$. Since $\W_{\H_1}^{(l)}$ is infinite, there exist pairwise disjoint open sets $\{U_m \}_{m \in \en}$ and points $\{ z_m\}_{m \in \en} \subseteq \W_{\H_1}^{(l)}$ such that $z_m \in U_m$ for each $m \in \en$. For a given $\zeta>0$ we find functions $\{g_m\}_{m \in \en} \subseteq \A_1^{E_1}$ such that for each $m \in \en$, 
\[0 \leq g_m \leq 1, g_m(z_m)>1-\frac{\zeta}{2^m} \text{ and } g_m <\frac{\zeta}{2^m} \text{ on } \Ch_{\H_1} K_1 \setminus U_m.\] Now, since $\B^{(k)} \cap \Omega_{\H_2}$ is finite,  it is enough to use Lemma \ref{finite} to obtain the function $h_1=g_m$ and the point $x_1=z_m$ for a suitable $m \in \en$.

Finally, for the case (iii) we suppose that $\W_{\H_1}^{(l)}=\{ x_1, \ldots, x_n\}$ and $\B^{(k)} \cap \Omega_{\H_2}=\{y_1, \ldots, y_m\}$. Given $\zeta>0$, we find pairwise disjoint open sets $U_1, \ldots, U_n$ such that $x_j \in U_j$ for each $j=1, \ldots, n$, and functions $h_1, \ldots, h_n$ in $\A_1^{E_1}$ satisfying that 
\[0 \leq h_j \leq 1, h_j(x_j)>1-\zeta \text{ and } h_j<\zeta \text{ on } \Ch_{\H_1} K_1 \setminus U_j, \quad j=1, \ldots, n.\]

As in the proof of Theorem \ref{jednostranna}c), we define a linear operator $S:(E_1^n, \norm{\cdot}_{\max}) \rightarrow \C(\B^{(k)} \cap \Omega_{\H_2}, E_2) \simeq (E_2^m, \norm{\cdot}_{\max})$ as
\[S(e_1, \ldots, e_n)=T(\sum_{i=1}^n h_i \otimes e_i)|_{(\B^{(k)} \cap \Omega_{\H_2})}.\] 
Then as in the proof of Theorem \ref{jednostranna}c), $S$ is bounded. Since by our assumption, $E_2^m$ does not contain an isomorphic copy of $E_1^n$, it follows that $S^{-1}$ is not bounded. Consequently, there exist nonzero vectors $e_1, \ldots, e_n \in E_1$ with $\max_{i=1, \ldots , n} \norm{e_i}=1$, such that $\norm{T(\sum_{i=1}^n h_i \otimes e_i)|_{(\B^{(k)} \cap \Omega_{\H_2})}} <\zeta$, and those are the vectors we were looking for. The proof is finished.

\end{proof}


\bibliography{iso-functions}\bibliographystyle{siam}
\end{document}